\documentclass[12pt]{amsart}
\usepackage{amssymb,amscd}
\usepackage[cp1251]{inputenc}

\usepackage{amsmath}
\usepackage{amsthm}
\usepackage{mathtext}
\usepackage{euscript}
\usepackage{graphics}

\usepackage[dvips]{graphicx}
\usepackage[all]{xy}

\input xypic

\linethickness{0.5pt}

\usepackage[english]{babel}

\theoremstyle{plain}

\newtheorem{theo}{Theorem}[section]
\newtheorem{prop}[theo]{Proposition}

\newtheorem{constr}[theo]{Construction}

\theoremstyle{definition}
\newtheorem*{defi}{Definition}
\newtheorem{exam}[theo]{Example}
\newtheorem{quest}[theo]{Question}

\theoremstyle{remark}
\newtheorem*{rema}{Remark}

\numberwithin{equation}{section}

\DeclareMathOperator{\Tor}{Tor}

\DeclareMathOperator{\rk}{rk}
\DeclareMathOperator{\link}{link}

\newcommand{\zp}{\mathcal Z_P}
\newcommand{\zk}{\mathcal Z_K}
\newcommand{\ko}{\Bbbk}
\DeclareMathOperator{\vc}{\mbox{\textit{vc}}}

\begin{document}

\title[Families of Golod complexes]{Families of minimally non-Golod complexes and their polyhedral products}

\author{Ivan Limonchenko}

\thanks{The work was supported by
RSF grant no. 14-11-00414.\\
}


\address{Department of Geometry and Topology,
Faculty of Mathematics and Mechanics, Moscow State University,
Leninskiye Gory, Moscow 119992, Russia}
\email{ilimonchenko@gmail.com}


\begin{abstract}

We consider families of simple polytopes $P$ and simplicial complexes $K$ well-known in polytope theory and convex geometry, and show that their moment-angle complexes have some remarkable homotopy properties which depend on combinatorics of the underlying complexes and algebraic properties of their Stanley--Reisner rings. We introduce infinite series of Golod and minimally non-Golod simplicial complexes $K$ with moment-angle complexes $\zk$ having free integral cohomology but not homotopy equivalent to a wedge of spheres or a connected sum of products of spheres respectively. We then prove a criterion for a simplicial multiwedge and composition of complexes to be Golod and minimally non-Golod and present a class of minimally non-Golod polytopal spheres.

\end{abstract}


\maketitle

\section{Introduction}
We denote by $K$ a simplicial complex of dimension $n-1$ on $m$ vertices and by $\ko$ a field or the ring of integers. Let $\ko[v_{1},\ldots,v_{m}]$ be the graded polynomial algebra on $m$ variables, $\deg(v_{i})=2$. The \emph{face ring} (or the
\emph{Stanley--Reisner ring}) of $K$ over $\ko$ is the quotient ring
$$
   \ko[K]=\ko[v_{1},\ldots,v_{m}]/\mathcal I_K
$$
where $\mathcal I_K$ is the ideal generated by those square free
monomials $v_{i_{1}}\cdots{v_{i_{k}}}$ for which
$\{i_{1},\ldots,i_{k}\}$ is not a simplex in $K$. We refer to
$\mathcal I_{K}$ as the \emph{Stanley--Reisner ideal} of~$K$.
Note that $\ko[K]$ is a $\ko$-algebra and a module over $\ko[v_{1},\ldots,{v_{m}}]$ via
the quotient projection. 

In what follows we denote by $P$ a simple $n$-dimensional convex polytope with $m$ {\textit{facets}} (i.e faces of codimension 1) $F_{1},\ldots,F_{m}$. Denote by $K_P$ the boundary $\partial P^*$ of the dual simplicial polytope. It can be viewed as a $(n-1)$-dimensional simplicial complex on the set $[m]$, whose simplices are subsets
$\{i_1,\ldots,i_k\}$ such that $F_{i_1}\cap\ldots\cap
F_{i_k}\ne\varnothing$ in~$P$.

Suppose $({\bf{X}},{\bf{A}})=\{(X_i,A_i)\}_{i=1}^{m}$ is a set of topological pairs.
A {\textit{polyhedral product}} is a topological space:
$$
({\bf{X}},{\bf{A}})^K=\bigcup\limits_{I\in K}({\bf{X}},{\bf{A}})^I,
$$
where $({\bf{X}},{\bf{A}})^I=\prod\limits_{i=1}^{m} Y_{i}$ for $Y_{i}=X_{i}$, if $i\in I$, and $Y_{i}=A_{i}$, if $i\notin I$.
Particular cases of a polyhedral product $({\bf{X}},{\bf{A}})^K$ are {\textit{moment-angle complexes}} $\zk=(\mathbb{D}^2,\mathbb{S}^1)^K$ and {\textit{real moment-angle complexes}} $\mathcal R_K=(\mathbb{D}^1,\mathbb{S}^0)^K$. We also call $\zp=\mathcal Z_{K_P}$ the {\textit{moment-angle manifold}} of $P$. By \cite[Corollary 6.2.5]{TT}, $\zp$ has a structure of a smooth manifold of
dimension $m+n$. 

The $\Tor$-groups of $\ko[K]$ acquire a topological interpretation by means of the following result on the cohomology of $\mathcal Z_K$.

\begin{theo}[{\cite[Theorem 4.5.4]{TT} or \cite[Theorem 4.7]{P}}]\label{zkcoh}
The cohomology algebra of the moment-angle complex $\mathcal Z_K$ is given
by the isomorphisms
\[
\begin{aligned}
  H^{*,*}(\mathcal Z_K;\ko)&\cong\Tor_{\ko[v_1,\ldots,v_m]}^{*,*}(\ko[K],\ko)\\
  &\cong H\bigl[\Lambda[u_1,\ldots,u_m]\otimes \ko[K],d\bigr]\\
  &\cong \bigoplus\limits_{I\subset [m]}\widetilde{H}^{*}(K_{I}),
\end{aligned}
\]
where bigrading and differential in the cohomology of the differential
bigraded algebra are defined by
\[
  \mathop{\mathrm{bideg}} u_i=(-1,2),\;\mathop{\mathrm{bideg}} v_i=(0,2);\quad
  du_i=v_i,\;dv_i=0.
\]
In the third row, $\widetilde{H}^*(K_{I})$ denotes the reduced simplicial cohomology of the {\textit{full subcomplex}} $K_{I}$ of $K$ (the restriction of $K$ to $I\subset [m]$). The last isomorphism is the sum of isomorphisms 
$$H^p(\mathcal Z_K)\cong\sum\limits_{I\subset [m]}\widetilde{H}^{p-|I|-1}(K_{I}),$$
and the ring structure is given by the maps
$$
\widetilde{H}^{p-|I|-1}(K_{I})\otimes\widetilde{H}^{q-|J|-1}(K_{J})\to \widetilde{H}^{p+q-|I|-|J|-1}(K_{I\cup J}),\eqno (*)
$$
which are induced by the canonical simplicial maps $K_{I\cup J}\hookrightarrow K_{I}*K_{J}$ (join of simplicial complexes) for $I\cap J=\varnothing$ and zero otherwise.
\end{theo}

Additively the following theorem due to Hochster holds.

\begin{theo}[{\cite{Hoch}}]\label{hoch}
For any simplicial complex $K$ on $m$ vertices we have:
$$
\Tor^{-i,2j}_{\ko[v_{1},\ldots,v_{m}]}(\ko[K],\ko)\cong\bigoplus\limits_{J\subset [m],\,|J|=j}\widetilde{H}^{j-i-1}(K_{J}).
$$
\end{theo}
The ranks of the bigraded components of the Tor-algebra 
$$
\beta^{-i,2j}(\ko[K])=\rk_{\ko}\Tor^{-i,2j}_{\ko[v_{1},\ldots,v_{m}]}(\ko[K],\ko)
$$ 
are called the {\textit{bigraded Betti numbers}} of $\ko[K]$ or $K$, when the base field $\ko$ is fixed.

A face ring $\ko[K]$ is called {\textit{Golod}}
if the multiplication and all higher Massey operations
in $\Tor_{\ko[v_{1},\ldots,{v_{m}}]}\bigl(\ko[K],\ko\bigr)$
are trivial. This property was first considered by Golod~\cite{G} in his study of local rings with rational Poincar\'e series, see also Gulliksen and Levin~\cite{G-L}. Due to the result of Berglund and J\"ollenbeck~\cite[Theorem 5.1]{B-J} $\ko[K]$ is Golod when the product in the Tor-algebra is trivial over a field $\ko$. We say that $K$ is a Golod complex when $\ko[K]$ is a Golod ring over any $\ko$.

If $K$ itself is not Golod but deleting any vertex $v$ from $K$
turns the restricted complex $K-v$ into a Golod one, then $\ko[K]$ and $K$ itself are called {\textit{minimally non-Golod}}.  

We need the next result due to Bahri, Bendersky, Cohen and Gitler which is true in a much more general situation of polyhedral products.

\begin{theo}[{\cite{BBCG}}]~\label{BBCGdecomp}
For any moment-angle complex its suspension $\Sigma\mathcal Z_K$ is homotopy equivalent to the wedge of suspensions over all non-simplex induced subcomplexes: $\bigvee_{J\notin K}\Sigma^{2+|J|}|K_J|.$
\end{theo}

The structure of this paper is as follows. We introduce several triangulations $K$ of classical manifolds with nice combinatorial properties for which we compute the homotopy types of $\zk$ and work out the Golod and minimally non-Golod properties of the complexes in Section 2. In particular, we prove Theorem~\ref{ComplexProjGolod} and Proposition~\ref{homotypeK(J)} which give infinite series of moment-angle complexes $\zk$ with free integral cohomology and $K$ being Golod or minimally non-Golod over any field, such that it is not homotopy equivalent to a wedge of spheres or a connected sum of products of spheres respectively (cf. Theorem~\ref{sKminnonGolod}).
In section 3 we prove a criterion when the simplicial multiwedge and composition of complexes are Golod and minimally non-Golod complexes (Theorem~\ref{SMWGolod} and Theorem~\ref{SubstGolod}). We use simplicial multiwedge construction to prove a criterion for minimal non-Golodness of the nerve complexes $K_P$ for simple polytopes $P$ with few facets (Theorem~\ref{FewVertices}). 

The author is grateful to Taras Panov for
many helpful discussions and advice. 
Many thanks to Anton Ayzenberg, Nickolay Erokhovets, Jelena Grbi\'c, Shizuo Kaji and Stephen Theriault for their comments and suggestions. The author would also like to thank the Institute of Mathematical Sciences and the organizers of the program on Combinatorial and Toric Homotopy in Singapore for providing excellent research conditions during the work on this paper.

\section{Some Golod complexes and their moment-angle complexes}

In this section we consider several well-known minimal triangulations of classical surfaces, as simplicial complexes $K$ for which we discuss the homotopy types and cohomology of $\zk$. We denote by $X^{\vee k}$ the $k$-fold wedge of $X$.

\begin{exam}~\label{ProjGolod}
Suppose $K$ is a 6-vertex minimal triangulation of $\mathbb{R}P^2$. 

Due to the result of Grbic, Panov, Theriault and Wu~\cite[Example 3.3]{G-P-T-W} $K$ is Golod and $\mathcal Z_K$ has a homotopy type of a wedge: 
$$
\zk\simeq (S^{5})^{\vee 10}\vee (S^{6})^{\vee 15}\vee (S^{7})^{\vee 6}\vee\Sigma^{7}\mathbb{R}P^{2}.
$$  
\end{exam}




\begin{prop}~\label{TorusGolod}
Suppose $K$ is a 7-vertex minimal triangulation of $\mathbb{T}^2$. Then $K$ is Golod and $\mathcal Z_K$ has a homotopy type of a wedge of spheres:
$$
\zk\simeq (S^{5})^{\vee 21}\vee (S^{6})^{\vee 49}\vee (S^{7})^{\vee 42}\vee (S^{8})^{\vee 14}\vee (S^{9})^{\vee 2}\vee S^{10}.
$$
\end{prop}
\begin{proof}
Let us prove Golodness of $K$. By Theorem~\ref{zkcoh}, the map 
$$
\widetilde{H}^{i}(K_{I})\otimes\widetilde{H}^{j}(K_{J})\to\widetilde{H}^{i+j+1}(K_{I\cup J})
$$
is trivial. Indeed, if $i=j=1$ then $H^{i+j+1}(K_{I\cup J})=0$ as $K$ has dimension $n-1=2$; if $i=0$ then $\widetilde{H}^{i}(K_I)=0$ as $K$ is 2-neighbourly.

Using Macaulay2 software~\cite{Mac}, we compute the bigraded Betti numbers of $K$ (over $\mathbb{Z}$). The tables of $\beta^{-i,2j}(K)$ in what follows have $n$ rows and $m$ columns. The number in the $k$s row and $l$s column equals
$\beta^{-l,2(l+k)}(K)$, where $ 1\leq {k}\leq {n}$ and $2\leq {l+k}\leq {m}$. Other bigraded Betti numbers are {\bf zero}, except for $\beta^{0,0}(K)=1$, see~\cite[Corollary 4.6.7]{TT}. The table below has $m=7$ columns and $n=3$ rows.

\begin{table}[h]
\begin{center}
{\small
\begin{tabular}{|c|c|c|c|c|c|c|}
\hline
0 & 0 & 0 & 0 & 0 & 0 & 0
\tabularnewline
\hline
21 & 49 & 42 & 14 & 2 & 0 & 0
\tabularnewline
\hline
0& 0 & 0 & 1 & 0 & 0 & 0
\tabularnewline
\hline
\end{tabular}
}
\end{center}
\caption{Bigraded Betti numbers of $\mathbb{T}^2_7$.}
\label{bBnT2}
\end{table}
To find the homotopy type of $\zk$ one can see that the stable homotopy decomposition in Theorem~\ref{BBCGdecomp} can be desuspended, since all the attaching maps of $k$-cells in the CW-complex $\zk$ in dimensions $6\leq k\leq 10$ are in the stable range. Then they are all null homotopic; the desuspension
in Theorem~\ref{BBCGdecomp} gives the homotopy type as in the statement required due to Theorem~\ref{hoch}, see Table~\ref{bBnT2}.
\end{proof}

The two above examples as well as all previously computed ones make it possible to ask the following question: is it true that if $K$ is a Golod complex and all its induced subcomplexes have free integral homology groups then $\mathcal Z_K$ is homotopy equivalent to a wedge of spheres?

The answer is negative as the next result shows.

\begin{exam}
Let $K$ be the minimal 9-vertex triangulation of $\mathbb{C}P^2$ which was described by K\"{u}hnel and Banchoff~\cite{K-B}. K\"{u}hnel and Lassmann~\cite{K-L} computed the symmetry group and proved combinatorial uniqueness and 3-neighbourness of $K$. 

Here $K$ is on the vertex set $\{0,\ldots,8\}$, and there are 36 maximal 4-dimensional faces which are given in the following table, see~\cite[p. 178]{K-L}:

\begin{table}[h]
\begin{center}
{\small
\begin{tabular}{|c|c|c|}
\hline
01234 & 70485 & 17562
\tabularnewline
\hline
01237 & 70481 & 17560
\tabularnewline
\hline
01267 & 70431 & 17580
\tabularnewline
\hline
02345 & 74852 & 15624
\tabularnewline
\hline
02367 & 74831 & 15680
\tabularnewline
\hline
03467 & 78531 & 16280
\tabularnewline
\hline
03456 & 78523 & 16248
\tabularnewline
\hline
04567 & 75231 & 12480
\tabularnewline
\hline
02358 & 74826 & 15643
\tabularnewline
\hline
02368 & 74836 & 15683
\tabularnewline
\hline
03568 & 78236 & 16483
\tabularnewline
\hline
02458 & 74526 & 15243
\tabularnewline
\hline
\end{tabular} 
}
\end{center}
\caption{Maximal simplices of $\mathbb{C}P^2_9$.}
\label{maxCP2}
\end{table}
The group $G$ of symmetries of $K$ has order 54, it acts transitively on the vertex set of $K$ and is generated by the 3 permutations $R,\,S,\,T$ (see~\cite[p. 179]{K-L}):
$$
R=(107)(245)(863),\,S=(128)(357),\,T=(28)(46)(53).
$$
\end{exam}

We note the following combinatorial property of $K$:

\begin{prop}\label{AlexDual}
$K$ is isomorphic to its Alexander dual complex $K^{\vee}$, that is the minimal non-faces of $K$ are exactly the complements to its maximal faces.
\end{prop}
\begin{proof}
Firstly, we prove that the complement to a maximal 4-face is a minimal non-face of $K$ and all minimal non-faces on 4 vertices appear in this way. The first part follows from the fact that the symmetry group $G$ sends (maximal) simplicies of $K$ to (maximal) simplicies of $K$ and acts transitively on this set, and for one of them, say for (01234) its complement (5678) is indeed a minimal non-face. The second part follows from the fact that the $f$-vector of $K$ is (9,36,84,90,36), see~\cite{K-L}, and $\binom{9}{4}=126=f_{3}+f_{4}$, therefore all the minimal non-faces on 4 vertices have the form above.

Suppose we have a minimal non-face of $K$ on 5 vertices: $I=(v_{1},\ldots,v_{5})$. Due to what is proved above, its complement $J$ is a 3-simplex of $K$ (otherwise, it should be a minimal non-face of $K$ by 3-neighbourness of $K$). By transitivity of the symmetry group action, take one such 3-face, say $(0123)$ (see Table~\ref{maxCP2}). But its complement is $(45678)$ and contains a non-face $(5678)$ (see above). We get a contradiction, so there are no minimal non-faces on 5 vertices in $K$.

Finally, suppose $I$ is a minimal non-face of $K$ on some $6$ vertices, as there are no simplices in $K$ more than on 5 vertices, $I=(v_{1},\ldots,v_{6})$. Then $(v_{i},v_{7},v_{8},v_{9})$ for $i=1,\ldots,6$ are minimal non-faces of $K$ on 4 vertices (as complements to maximal faces). But $K$ is 3-neighbourly and pure, so $(v_{7},v_{8},v_{9})$ should be in one of the maximal 4-simplices of $K$. We get a contradiction and thus we found all minimal non-faces of $K$.
\end{proof}

\begin{rema}
Note, that the same is true for $K=\mathbb{R}P^2_6$: $K$ is combinatorially equivalent to its Alexander dual complex $K^{\vee}$, but is not true for $\mathbb{T}^2_7$.
\end{rema}

\begin{theo}~\label{ComplexProjGolod}
The 9-vertex minimal triangulation $K$ of $\mathbb{C}P^2$ is a Golod complex, all its induced subcomplexes have free integral homology groups and $\mathcal Z_K$ has a homotopy type of the following suspension:
$$
\zk\simeq (S^{7})^{\vee 36}\vee (S^{8})^{\vee 90}\vee (S^{9})^{\vee 84}\vee (S^{10})^{\vee 36}\vee (S^{11})^{\vee 9}\vee\Sigma^{10}\mathbb{C}P^2.
$$
\end{theo}
\begin{proof}
We first prove Golodness of $K$. Consider the cup-product in the Tor-algebra of $K$ induced by a simplicial embedding of full subcomplexes on some vertex sets $I$ and $J$:
$$
\widetilde{H}^{i}(K_{I})\otimes\widetilde{H}^{j}(K_{J})\to\widetilde{H}^{i+j+1}(K_{I\cup J})
$$
One has the following cases.
\begin{enumerate}
\item $I$ and $J$ are 4-vertex minimal non-faces and $i=j=2$;
\item $|I|=4,|J|=5,i=1$ (then $I\sqcup J=[9]$). 
\end{enumerate} 
The first is impossible as dimension of $K$ equals $n-1=4$. The second is impossible by Proposition~\ref{AlexDual}.
Therefore, $K=\mathbb{C}P^2_9$ is a Golod complex.

As in Proposition~\ref{TorusGolod} we compute bigraded Betti numbers of $K$ using Macaulay2 program. The following table has $m=9$ columns and $n=5$ rows.

\begin{table}[h]
\begin{center}
{\small
\begin{tabular}{|c|c|c|c|c|c|c|c|c|}
\hline
0 & 0 & 0 & 0 & 0 & 0 & 0 & 0 & 0
\tabularnewline
\hline
0 & 0 & 0 & 0 & 0 & 0 & 0 & 0 & 0
\tabularnewline
\hline
36 & 90 & 84 & 36 & 9 & 1 & 0 & 0 & 0
\tabularnewline
\hline
0 & 0 & 0 & 0 & 0 & 0 & 0 & 0 & 0
\tabularnewline
\hline
0 & 0 & 0 & 1 & 0 & 0 & 0 & 0 & 0 
\tabularnewline
\hline
\end{tabular} 
}
\end{center}
\caption{Bigraded Betti numbers of $\mathbb{C}P^2_9$.}
\label{bBnCP2}
\end{table}
To find the homotopy type of $\zk$ one can see that the stable homotopy decomposition in Theorem~\ref{BBCGdecomp} can be desuspended, since all the attaching maps of $k$-cells in the CW-complex $\zk$ in dimensions $8\leq k\leq 14$ are in the stable range. Then they are all null homotopic; the desuspension
in Theorem~\ref{BBCGdecomp} gives the homotopy type as in the statement required due to Theorem~\ref{hoch}, see Table~\ref{bBnCP2}.
\end{proof}

\begin{rema}
The desuspension in Theorem~\ref{BBCGdecomp} for the cases of Proposition~\ref{TorusGolod} and Theorem~\ref{ComplexProjGolod} follows also from the triviality of the fat wedge filtration of $\mathbb{R}\mathcal{Z}_K$ regarding neighbourness of these triangulations, due to the result of Iriye and Kishimoto~\cite[Theorem 10.9]{I-K}. Another argument for this is true and a different homotopy theoretical approach to the Golod property for $K$ related to the co-H-space case for moment-angle complexes $\zk$ can be found in the work of Beben and Grbi\'c~\cite{B-G}.
\end{rema}

Now we turn to the class of minimally non-Golod complexes $K$ and their moment-angle complexes $\zk$. We denote by $sK$ the stellar subdivision of $K$ at a maximal simplex $\sigma$. Geometrically $sK$ is obtained from $K$ by replacing $\sigma$ with a cone over its boundary; we denote the cone vertex by $v$.

\begin{theo}~\label{sKminnonGolod}
The following statements hold:
\begin{itemize}
\item[(a)] If $K=\mathbb{R}P^{2}_{6}$ then $sK$ is Golod over any field $\ko$, except for $char(\ko)=2$. In the latter case $sK$ is minimally non-Golod.
\item[(b)] If $K=\mathbb{T}^{2}_{7}$ then $sK$ is minimally non-Golod.
\item[(c)] If $K=\mathbb{C}P^{2}_{9}$ then $sK$ is minimally non-Golod.
\end{itemize}
Moreover, in all these three cases $\mathcal Z_{sK}$ is not homotopy equivalent to any connected sum of products of spheres. 
\end{theo}
\begin{proof}
Let us prove statement (a).\\
Suppose the new vertex 7 is the vertex $v$ of a cone over the facet $(456)$. As $K=\mathbb{R}P^2_6$ is a 2-neighbourly complex of dimension $n-1=2$, the only 
nontrivial product in the Tor-algebra of $sK$ is of the form:
$$
\widetilde{H}^{0}(K_I;\ko)\otimes\widetilde{H}^{1}(K_J;\ko)\to\widetilde{H}^{2}(\mathbb{R}P^2;\ko),
$$
where $I\sqcup J=[7]$ and $7\in I$. We have $\widetilde{H}^{2}(\mathbb{R}P^2;\ko)\neq 0$ if and only if $char(\ko)=2$, which finishes the proof in this case.

Let us prove statement (b).\\
Suppose the new vertex is 8 over the facet $(123)$. Observe that the complex $sK$ is non-Golod as we have a nontrivial product:
$$
\widetilde{H}^{0}(sK_I)\otimes\widetilde{H}^{1}(sK_J)\to\widetilde{H}^{2}(\mathbb{T}^2),
$$
where $I\sqcup J=[7]$ and $7\in I$. If we delete a vertex $v$ from $sK$ then there are no more 2-dim cohomology classes, and the 2-neighbourness of $K$ implies that for $K'=sK-v$: $\widetilde{H}^{0}(K'_I)$ and $\widetilde{H}^{0}(K'_J)$ can not be nonzero simultaneously when $I\cap J=\varnothing$. Therefore, the product in Tor-algebra of the induced complex $K'$ is trivial and $sK$ is minimally non-Golod.

Finally, we prove statement (c).\\
Suppose, $K$ is on the vertices $\{0,\ldots,8\}$ and 9 is the vertex of the cone over the facet $(01234)$. 
The complex $sK$ is non-Golod by the same reason as in the previous case.
To prove that $sK$ is minimally non-Golod we consider the complex $K'$ obtained by deleting vertex $v$ from $K$. There are 3 cases:

1) $v\in\{0,1,2,3,4\}$. Then we have:
$$
sK-v=(K-v)\cup_{\Delta^3}\Delta^4.
$$ 
This is a Golod complex by~\cite[Proposition 3.1]{Li2}. 

2) $v\in\{5,6,7,8\}$. According to the above description of the symmetry group $G$, it is enough to consider the case $v=5$. We first compute the bigraded Betti numbers of $K'=sK-v$:

\begin{table}[h]
\begin{center}
{\small
\begin{tabular}{|c|c|c|c|c|c|c|c|c|}
\hline
8 & 28 & 56 & 70 & 56 & 28 & 8 & 1 & 0
\tabularnewline
\hline
0 & 0 & 0 & 0 & 0 & 0 & 0 & 0 & 0
\tabularnewline
\hline
20 & 60 & 68 & 36 & 9 & 1 & 0 & 0 & 0
\tabularnewline
\hline
1 & 4 & 6 & 4 & 1 & 0 & 0 & 0 & 0
\tabularnewline
\hline
0 & 0 & 0 & 0 & 0 & 0 & 0 & 0 & 0
\tabularnewline
\hline
\end{tabular} 
}
\end{center}
\caption{Bigraded Betti numbers of $K'$.}
\label{bBnK'}
\end{table}

Due to Theorem~\ref{hoch} this means that $H^{1}(K'_I)=0$ for all full subcomplexes on $I$ vertices in $K'$. If $\widetilde{H}^{0}(K'_I)\neq 0$ then $9\in I$. Applying the same argument to the following table of bigraded Betti numbers of $K'-9$, one can see that $H^{2}(K'_J)=0$ for all $J$ with $9\notin J$:

\begin{table}[h]
\begin{center}
{\small
\begin{tabular}{|c|c|c|c|c|c|c|c|}
\hline
1 & 0 & 0 & 0 & 0 & 0 & 0 & 0
\tabularnewline
\hline
0 & 0 & 0 & 0 & 0 & 0 & 0 & 0
\tabularnewline
\hline
0 & 0 & 0 & 0 & 0 & 0 & 0 & 0
\tabularnewline
\hline
20 & 40 & 33 & 14 & 2 & 0 & 0 & 0 
\tabularnewline
\hline
1 & 8 & 9 & 2 & 0 & 0 & 0 & 0
\tabularnewline
\hline
\end{tabular} 
}
\end{center}
\caption{Bigraded Betti numbers of $K'-9$.}
\label{bBnK'min9}
\end{table}

Thus, the cup-product in the Tor-algebra of $K'=sK-v$ is trivial. \\
3) $v$ is a new vertex 9. Then $|K'|\cong\mathbb{C}P^2-D^4\simeq S^2$, and the neighbourness of $K$ implies that $\widetilde{H}^{0}(K'_I)=0$ for all $I$, so the product in the Tor-algebra of $K'=sK-v$ is trivial.\\
Therefore, $sK$ is minimally non-Golod for $K=\mathbb{C}P^2_9$.

As for the last statement of the theorem, note that $K$ is not a Gorenstein* complex (and not even Cohen-Macaulay) in either of the three cases. By the Avramov-Golod theorem (see~\cite[Theorem 3.4.4]{TT}) its Tor-algebra is not a Poincar\'e algebra. Then by Theorem~\ref{zkcoh}, $\zk$ can not be homotopy equivalent to a closed oriented manifold, nor is it homotopy equivalent to a connected sum of sphere products. 
\end{proof}

The cases 2) and 3) in Theorem~\ref{sKminnonGolod} provide counterexamples to a question raised in~\cite{Li2}, and we therefore give its following modified version:
\begin{quest}
Assume $K$ is a triangulated sphere. Then $\mathcal Z_K$ is topologically equivalent to a connected sum of sphere products with two spheres in each product if and only if $K$ is minimally non-Golod and torsion free (that is $H^{*}(\zk)$ is a free group).
\end{quest}
The ``only if'' statement is true:
\begin{prop}
Suppose $K$ is a triangulated sphere. Then the following holds:
\begin{itemize}
\item[(a)] If $\zk$ is homotopy equivalent to a connected sum of products of spheres with two spheres in each product then $K$ is minimally non-Golod and torsion free.
\item[(b)] $\zk$ has cup-length 2 if and only if $K$ is minimally non-Golod.
\end{itemize} 
\end{prop}
\begin{proof}
For statement (a) Theorem~\ref{zkcoh} implies that the only nontrivial cup-product in the Tor-algebra of $K$ arises from complementary full subcomplexes $K_I$ and $K_{J}$, $I\sqcup J=[m]$, and is equal to the generator class in $H^{n-1}(K);\ko$. The latter is equivalent for triangulated spheres to $K$ being minimally non-Golod; $K$ is obviously torsion free as $H^{*}(\zk)$ is a free group for a connected sum of sphere products.

The proof of statement (b) is straightforward by Theorem~\ref{zkcoh} and Alexander duality. 
\end{proof}

\section{Minimally non-Golod complexes, simple polytopes and polyhedral products}

We begin with the definition of the simplicial multiwedge construction due to Bahri, Bendersky, Cohen and Gitler~\cite{BBCGit}.

\begin{defi}
Let $K$ be an $(n-1)$-dimensional simplicial complex on $m$ vertices on the vertex set $\{v_{1},\ldots,v_{m}\}$, and let $J=(j_{1},\ldots,j_{m})$ be a sequence of positive integers. Then {\it the simplicial multiwedge} $K(J)$ is the complex on $j_{1}+\ldots+j_{m}$ vertices whose minimal non-faces have the following form:
$$
\{v_{i_{1}1},\ldots,v_{i_{1}j_{i_1}},\ldots,v_{i_{k}1},\ldots,v_{i_{k}j_{i_k}}\},
$$  
where $\{v_{i_1},\ldots,v_{i_k}\}$ is a minimal non-face of $K$.
\end{defi}

The definition provides an explicit description of the Stanley--Reisner ideal of $K(J)$. Obviously, $K(1,\ldots,1)=K$.

\begin{rema}
If $K=K_P$ is the nerve complex of a simple $n$-polytope, then there is a simple polytope $P(J)$ satisfying $K_{P(J)}=K_P(J)$, see~\cite{BBCGit}. Obviously, $P(J)$ has $j_{1}+\ldots+j_{m}$ facets and is of dimension $(j_{1}-1)+\ldots+(j_{m}-1)+n$, so the number $m-n$ is preserved by simplicial multiwedge operation.
\end{rema}

\begin{exam}
Let $P$ be a 6-gon. Then $P(2,1,1,1,1,1)$ is a simple 3-polytope with 7 facets. It is easy to see that it is a truncation polytope: $P(J)=\vc^{3}(\Delta^{3})$.
\end{exam}

In what follows we need the following result about polyhedral products of simplicial multiwedges.

\begin{theo}[{\cite[Theorem 7.5, Corollary 7.6]{BBCGit}}]\label{BBCGcohring}
There is an action of $\mathbb{T}^m$ on $(\underline{D}^{2J},\underline{S}^{2J-1})^K$ and on $(\mathbb{D}^2,\mathbb{S}^1)^K$ with respect to which they are equivariantly homeomorphic.\\
This yields the spaces $(\mathbb{D}^2,\mathbb{S}^1)^{K(J)}$ and $(\mathbb{D}^2,\mathbb{S}^1)^K$ have isomorphic ungraded cohomology rings.
\end{theo}

The simplicial multiwedge construction preserves the Golod and minimal non-Golod properties of simplicial complexes:

\begin{prop}\label{SMWGolod}
Let $J=(j_{1},\ldots,j_{m})$ be a sequence of positive integers:
\begin{itemize}
\item[(a)] $K$ is a Golod complex if and only if $K(J)$ is Golod.
\item[(b)] $K$ is a minimally non-Golod complex if and only if $K(J)$ is minimally non-Golod.
\end{itemize}
\end{prop}
\begin{proof}
Statement (a) follows directly from Theorem~\ref{BBCGcohring}. 
We prove statement (b).
Suppose $K$ is minimally non-Golod. Then $K(J)$ is non-Golod by statement (a). Consider the complex obtained by removing a vertex $v$ from $K(J)$. 
Let $v=v_{ij_{l}}$, for $1\leq i\leq m, 1\leq l\leq i$ in the notation from the definition of the simplicial multiwedge. Then 
$$
K(J)-v=(K-i)(j_{1},\ldots,j_{i-1},j_{i+1},\ldots,j_{m})
$$ 
is a Golod complex by statement (a).
 
Therefore, we proved that $K(J)$ is minimally non-Golod if $K$ is minimally non-Golod. The opposite statement is proved in the same way.
\end{proof}

\begin{prop}\label{homotypeK(J)}
Let $K=\mathbb{C}P^{2}_{9}(J)$ for a sequence of positive integers $J$. Then $\zk$ is not homotopy equivalent to a wedge of spheres. \\
Moreover, for any simplicial complex $K$, if $\zk$ is not homotopy equivalent to a wedge of spheres, then so is $\mathcal Z_{K(J)}$.
\end{prop}
\begin{proof}
In the case $J=(1,\ldots,1)$ we know the homotopy type of $\zk$ (see Theorem~\ref{ComplexProjGolod}). The Steenrod square $Sq^2$ is nonzero in cohomology of any suspension over $\mathbb{C}P^2$ and all cohomology operations are trivial on wedges of spheres. 

In general, the isomorphism of ungraded cohomology rings from Theorem~\ref{BBCGcohring} with coeffiecients in $\mathbb{Z}/p$ ($p$ is prime) is an isomorphism of $\mathbb{Z}/p$-modules commuting with the action of the Steenrod algebra (see~\cite[Corollary 7.7]{BBCGit}), thus $\zk$ cannot be homotopy equivalent to a wedge of spheres.
\end{proof}

\begin{rema}
Note that, in general, $K(J)$ for a neighbourly complex $K$ is no longer a neighbourly simplicial complex. 
\end{rema}

Next we study minimally non-Golodness for $(n-1)$-dimensional spheres with few vertices. Note that any $(n-1)$-dimensional sphere with $m\leq n+3$ vertices is polytopal.

\begin{theo}\label{FewVertices}
Suppose $K=K_P$ is a nerve complex of a simple $n$-polytope with $m\leq n+3$ facets. The following statements hold:
\begin{itemize}
\item[(a)] If $m=n+1$ then $P$ is a simplex, $\mathcal Z_P$ is a sphere and $K_P$ is Golod.
\item[(b)] If $m=n+2$ then $P$ is combinatorially a product of two simplices, $\mathcal Z_P$ is a product of two odd-dimensional spheres and $K_P$ is minimally non-Golod.
\item[(c)] If $m=n+3$ then $K_P$ is minimally non-Golod if and only if $P$ is {\bf{not}} a product of 3 simplices.
\end{itemize}
\end{theo}
\begin{proof}
The statements (a) and (b) are obvious since $n$-polytopes with $m=n+1$ or $m=n+2$ are determined uniquely up to affine or projective equivalence, respectively.\\
Let us prove the statement (c).\\
According to a result of Erokhovets~\cite[Theorem 2.3.48]{Er}, if $P$ is a simple $n$-polytope with $m=n+3$ facets, then $P=C^{2k-4}(2k-1)^{*}(j_1,\dots,j_{2k-1})$, where $k\geqslant 3$ and $C^{n}(m)$ denotes a $n$-dimensional cyclic polytope with $m$ vertices. 

A cyclic polytope is neighbourly (see~\cite{Zieg}). By~\cite[Proposition 3.6]{Li2} the nerve complex $K_P$ of an even dimensional dual neighbouly polytope $P$ is minimally non-Golod. The proof is finished applying Proposition~\ref{SMWGolod}.
\end{proof}

The topological types of the moment-angle manifolds $\zp$ corresponding to the statement (c) of Theorem~\ref{FewVertices} are described as follows.

\begin{prop}[\cite{LdM},\cite{Er}]
Let $P$ be a simple $n$-polytope with $m=n+3$ vertices, so that $P=C^{2k-4}(2k-1)^{*}(j_1,\dots,j_{2k-1})$, for some $k\geqslant 3$. Then
$$
\mathcal{Z}_P\cong\mathop{\rm \#}\limits_{i=1}^{2k-1}S^{2\varphi_i-1}\times
S^{2\psi_{i+k-1}-2},
$$
where $\varphi_r=j_{r}+\ldots+j_{r+k-2}$, $\psi_{r}=j_{r}+\ldots+j_{r+k-1}$,
and all the indices are considered modulo $2k-1$.
\end{prop}

We therefore obtain infinite families of triangulated spheres $K$ which are minimally non-Golod and torsion free, whose corresponding $\zk$ are connected sums of products of two spheres.

\begin{rema}
Minimally non-Golodness in statement (c) of Theorem~\ref{FewVertices} can be also deduced from the explicit description of the multiplication in $H^{*,*}(\mathcal{Z}_P)$ in the case $m=n+3$, see~\cite[Theorem 2.5.8]{Er}
\end{rema}

We can also extend our results by considering the operation of composition of simplicial complexes, originally defined by Ayzenberg~\cite{Ay}.

\begin{defi}
Suppose $K$ is an $(n-1)$-dimensional simplicial complex on $m$ vertices, $K_1,\ldots,K_{m}$ are simplicial complexes (may be empty or with ghost vertices) on the sets $[l_1],\ldots,[l_m]$ respectively. Then the \textit{composition} of $K$ with $K_i, 1\leq i\leq m$ is the simplicial complex $K(K_1,\ldots,K_m)$ on the set $[l_1]\sqcup\ldots\sqcup [l_m]$ defined as follows: a set $I=I_1\sqcup\ldots\sqcup I_{m}$, with $I_{j}\subset [l_j]$ is a simplex of $K(K_1,\ldots,K_m)$ if and only if $\{j\in[m]| I_j\notin K_j\}\in K$.
The composition complex has therefore $l_{1}+\ldots+l_{m}$ vertices and dimension $(n-1)+(l_{1}+\ldots+l_{m}-m)$.
\end{defi}

\begin{exam}
Let $K_i=\partial\Delta^{j_{i}-1}$ for $1\leq i\leq m$. Then $K(K_{1},\ldots,K_{m})=K(J)$ is the simplicial multiwedge.
\end{exam}

\begin{theo}\label{SubstGolod}
Let $K_1,\ldots,K_m$ be simplicial complexes.
\begin{itemize}
\item[(a)] $K(K_1,\ldots,K_m)$ is Golod if and only if $K_{[m]-\{s_{1},\ldots,s_{r}\}}$ is Golod, where $1\leq s_{1},\ldots,s_{r}\leq m$ are such that $K_{s_{i}}=\Delta^{l_{s_i}-1}$ for all $1\leq i\leq r$ ($l_{i}\geq 1$).
\item[(b)] $K(K_1,\ldots,K_m)$ is minimally non-Golod if and only if $K_{s_{i}}=\Delta^{l_{s_i}-1}$ for $1\leq i\leq r$ and $K_{j}=\partial\Delta^{l_{j}-1}$ for $j\neq s_{i}, 1\leq i\leq r$ and $K_{[m]-\{s_{1},\ldots,s_{r}\}}$ is minimally non-Golod ($l_{i}\geq 1$).
\end{itemize}
\end{theo}
\begin{proof}
We proceed by induction on the number $N$ of non-empty complexes in $K_{1},\ldots,K_{m}$. The base case $N=0$ is trivial, so consider $N=1$. Suppose that $K_i=L$ is the only non-empty complex. Then
\begin{equation}\label{KL}
K(L)=K(\varnothing_{1},\ldots,K_i,\ldots,\varnothing_{m})=L*(K-i)\cup\Delta^{l_{i}-1}*\link_{i}K,
\end{equation}
where the simplicial complexes in the union are glued along their common subcomplex $L*\link_{i}K$.\\
Let us prove (a).\\
For the ``only if'' part, if $L$ is a simplex then $K(L)=\Delta^{l_{i}-1}*(K-i)$ and the statement is true as $\mathcal Z_{K(L)}\simeq\mathcal Z_{K-i}$. Otherwise, take a minimal non-face $V$ in $L$. Then $K(V)$ is a full subcomplex in a Golod complex $K(L)$ and thus $K(V)$ is Golod, therefore $K$ is Golod by Proposition~\ref{SMWGolod}.

For the ``if'' part, if $L$ is a simplex then $K(L)=\Delta^{l_{i}-1}*(K-i)$ and the statement is true. Otherwise, suppose $K(L)$ is non-Golod and the following map is nontrivial:
$$
\widetilde{H}^{i}(K(L)_I)\otimes\widetilde{H}^{j}(K(L)_J)\to\widetilde{H}^{i+j+1}(K(L)_{I\sqcup J}).
$$
Then it is also nontrivial viewed as a cup-product in the Tor-algebra of $K(\partial\Delta^{(I\sqcup J)\cap L})$, but the latter complex is Golod by Proposition~\ref{SMWGolod}.

Let us prove (b).\\
The ``if'' part follows obviously from Proposition~\ref{SMWGolod}. Now we prove the ``only if'' part. Suppose that $K(L)$ is minimally non-Golod. If $L$ is a simplex, then $\mathcal Z_{K(L)}\simeq\mathcal Z_{K-i}$; if $L$ is the boundary of a simplex then $K$ is minimally non-Golod by Proposition~\ref{SMWGolod} and the statement is true. Assume that $L$ is neither a simplex nor the boundary of a simplex. Then there is a proper subset of vertices $V\subset L$ (a minimal non-face of $L$) such that $L_V$ is the boundary of a simplex. Note that $K(L_V)$ is a full subcomplex in $K(L)$, so that $K(L_V)$ is Golod, as $K(L)$ is minimally non-Golod. Then $K$ is Golod by Proposition~\ref{SMWGolod} and we get a contradiction with part (a).

To make an induction step in both (a) and (b) we use the following result~\cite[Corollary 4.14]{Ay}:
$$
K(K_1,\ldots,K_m)=K(L)(K_1,\ldots,K_{i-1},\varnothing,\ldots,\varnothing,K_{i+1},\ldots,K_m),
$$
where there are exactly $l_i$ empty simplicial complexes in the second substitution. This finishes the proof by induction on the number of non-empty complexes in the composition of simplicial complexes.
\end{proof}

We finish the article by introducing a characterization of minimally non-Golod complexes $K$ in terms of cohomology length $cup(\zk)$ of their moment-angle complexes $\zk$. In fact, we show that for a minimally non-Golod complex $K$ the combinatorial invariant $cup(\zk)$ is equal to either 1, or 2, and both possibilities may occur. Moreover, we introduce a 4-dimensional minimally non-Golod simplicial complex $\mathcal K$ on 9 vertices such that $H^*(\mathcal Z_{\mathcal K})$ contains a non-trivial triple Massey product. This complex $\mathcal K$ is defined below.

\begin{constr}
Let us determine a simplicial complex $\mathcal K$ from the set of its minimal non-faces. That is, we have the following description of its Stanley--Reisner ring (over a field $\ko$):
$$
\ko[\mathcal K]=\ko[v_{1},\ldots,v_{9}]/I_{\mathcal K},
$$
where $I_{\mathcal K}$ is generated by the 15 monomials: 
$$
v_{1}v_{2}v_{3},v_{4}v_{5}v_{6},v_{7}v_{8}v_{9},v_{1}v_{2}v_{4}v_{5},v_{5}v_{6}v_{7}v_{8},
$$
$$
v_{2}v_{3}v_{7}v_{8},v_{1}v_{4}v_{7},v_{2}v_{3}v_{5}v_{6}v_{7},v_{1}v_{2}v_{4}v_{6}v_{8}v_{9},v_{1}v_{3}v_{4}v_{5}v_{8}v_{9},
$$
$$
v_{1}v_{3}v_{5}v_{6}v_{7}v_{9},v_{2}v_{3}v_{4}v_{5}v_{7}v_{9},v_{2}v_{3}v_{4}v_{5}v_{8}v_{9},v_{2}v_{3}v_{4}v_{6}v_{7}v_{9},v_{2}v_{3}v_{5}v_{6}v_{8}v_{9}.
$$
It is easy to see that $\dim\mathcal K=4$ and $m=9$, and therefore, $\dim\mathcal Z_{\mathcal K}=14$.
\end{constr}

\begin{theo}
The next statements hold.
\begin{itemize}
\item[(a)] If $K$ is a minimally non-Golod complex, then $cup(\zk)\leq 2$; 
\item[(b)] $\mathcal K$ is a minimally non-Golod complex such that 
\begin{itemize}
\item $cup(\mathcal Z_{\mathcal K})=1$;
\item There is a non-trivial triple Massey product $\langle [e_{1}],[e_{2}],[e_{3}]\rangle\in H^{14}(\mathcal Z_{\mathcal K})$, where the basis elements in the Taylor complex of $\mathcal K$ are indexed according to the above order in the minimal set of generators of $I_{\mathcal K}$. 
\end{itemize}
\end{itemize}
\end{theo}
\begin{proof}
To prove statement (a), suppose, on the contrary, that there exists a minimally non-Golod complex $K$ such that there is a non-zero product of 3 elements in $H^*(\zk;\ko)$, where $\ko$ is a certain field. By Theorem~\ref{zkcoh}, it follows that there is a non-zero product of three cohomology classes represented by certain simplicial cocycles $a_{1},a_{2}$, and $a_{3}$ on the vertex subsets $I_{1},I_{2}$, and $I_{3}$ of $[m]$. In particular, these subsets are pairwisely disjoint and the corresponding 3 cohomology classes have pairwisely non-trivial cup products. However, since $K$ is minimally non-Golod over $\ko$, the cup product of $[a_{1}]$ and $[a_{2}]$ should be zero as a product of two positive-degree elements in $H^*(\mathcal K_{[m]\backslash v})$, for $v\in I_{3}$ and $K_{[m]\backslash v}$ being a Golod complex (over $\ko$). We got a contradiction which finishes the proof of the first statement.

To prove statement (b), note that $\mathcal Z_{\mathcal K}$ has cohomology length 1 and its cohomology contains the above non-trivial triple Massey product of elements of degree 5 due to~\cite[Theorem 3.1, Remark 3.3]{Kat}. From the latter it follows immediately that $\mathcal K$ is not Golod. However, for each vertex $v\in\mathcal K$ the full subcomplex $K_{[m]\backslash v}$ has 8 vertices and the cup product in $H^*(\mathcal K_{[m]\backslash v})$ is trivial due to Theorem~\ref{zkcoh}. By~\cite[Theorem 6.3(5)]{Kat}, $K_{[m]\backslash v}$ is a Golod complex and therefore we proved that $\mathcal K$ is minimally non-Golod. This finishes the proof of the theorem.  
\end{proof}

\end{document}